\documentclass{amsart}
\usepackage{graphicx} 
\usepackage{amsmath}
\usepackage{amssymb}
\usepackage{amsthm}
\usepackage{thmtools}
\usepackage{tikz-cd}
\usepackage{dirtytalk}
\usepackage{xcolor}
\usepackage[shortlabels]{enumitem}
\usepackage[english]{babel}

\usepackage{hyperref}
\usepackage{cleveref}

\numberwithin{equation}{section}

\newtheorem{alphathm}{Theorem}
\newtheorem{theorem}{Theorem}[section]
\newtheorem{lemma}[theorem]{Lemma}
\newtheorem{corollary}[theorem]{Corollary}
\newtheorem{proposition}[theorem]{Proposition}
\theoremstyle{definition}
\newtheorem{alphadef}[alphathm]{Definition}
\newtheorem{definition}[theorem]{Definition}
\newtheorem{example}[theorem]{Example}
\newtheorem*{ack}{Acknowledgements}
\theoremstyle{remark}
\newtheorem{remark}[theorem]{Remark}

\DeclareMathOperator{\Iso}{Iso}
\DeclareMathOperator{\Aut}{Aut}
\DeclareMathOperator{\Inn}{Inn}

\DeclareMathOperator{\Ad}{Ad}

\newcommand{\R}[0]{\mathbb{R}}
\newcommand{\C}[0]{\mathbb{C}}
\newcommand{\T}[0]{\mathbb{T}}
\newcommand{\M}[0]{\mathcal{M}}
\newcommand{\N}[0]{\mathcal{N}}
\newcommand{\U}[0]{\mathcal{U}}
\renewcommand{\H}[0]{\mathcal{H}}

\title[Non-locally trivial $\mathrm{W}^*$-bundles]{A non-locally trivial $\mathrm{W}^*$-bundle with fixed factorial fibres}
\author{Kiefer Mommaerts}
\thanks{This work was funded by the Carlsberg Foundation grant CF24-1144.}
\address{Department of  Mathematics and  Computer Science,
University of Southern Denmark, Campusvej 55, 5230 Odense, Denmark}
\email{kiefer@imada.sdu.dk}
\date{\today}

\begin{document}
\begin{abstract}
   In this paper we construct the first example of a non-locally trivial $\mathrm{W}^*$-bundle whose fibres are all isomorphic to some fixed $\mathrm{II}_1$ factor. This is achieved by introducing a notion of uniformly having spectral gap for $\mathrm{W}^*$-bundles. For bundles with fixed factorial fibres, the negation of having this uniform spectral gap property provides an obstruction for being locally trivial. This results in seemingly elementary examples of $\mathrm{W}^*$-bundles whose fibres are all isomorphic to some fixed factor but that are not locally trivial, even over spaces with covering dimension equal to zero.
\end{abstract}
\maketitle

\section{Introduction}
Tracially continuous $\mathrm{W}^*$-bundles were introduced by Ozawa in \cite[Section 5]{OzawaWbundle}. They provide an operator algebraic framework for bundles of tracial von Neumann algebras indexed over some compact Hausdorff space. The most elementary examples of $\mathrm{W}^*$-bundles are the trivial bundles, which are precisely those of the following form
\begin{equation}
    \label{eq:trivialbundle}
    C_\sigma(X,M) := \{f: X\to M\mid f\text{ is } \|\cdot\|\text{-bounded and } \|\cdot\|_{2,\tau}\text{-continuous}\}.
\end{equation}
Here $X$ is a compact Hausdorff space and $(M,\tau)$ a tracial von Neumann algebra.

Central to the theory are \emph{triviality results}, which give sufficient conditions on a $\mathrm{W}^*$-bundle $\M$ that force it to be trivial. Of particular interest is the question of whether or not every $\mathrm{W}^*$-bundle whose fibres are all isomorphic to the hyperfinite $\mathrm{II}_1$ factor $R$ is automatically trivial. This is due to its intimate ties with the Toms--Winter conjecture, see \cite[Section 1.6]{2024traciallycompletecalgebras}.

Already in \cite[Theorem 15]{OzawaWbundle} such a triviality theorem is presented, in contemporary terms this result can be stated as follows (modulo some separability assumptions): A $\mathrm{W^*}$-bundle $\M$ with all fibres isomorphic to $R$ is trivial if and only if it has the uniform property Gamma of \cite{unifGamma}; see also \cite{2024traciallycompletecalgebras}.

Later, in \cite{EvingtonPennig}, Evington and Pennig provided a more abstract framework for studying $\mathrm{W}^*$-bundles. One of the main advantages of this point of view is that it allows for a natural definition of the restriction $\M\lvert_Y$ of a $\mathrm{W}^*$-bundle $\M$ over $X$ to any closed subset $Y\subseteq X$. Using this construction Evington and Pennig give the definition of a \emph{locally trivial $\mathrm{W}^*$-bundle}, which are heuristically the bundles that locally look like \eqref{eq:trivialbundle}. Subsequently, Evington and Pennig classify locally trivial $\mathrm{W}^*$-bundles with all fibres isomorphic to some fixed $\mathrm{II}_1$ factor $M$ using topological methods. That is, they give a one to one correspondence between the isomorphism classes of such locally trivial $\mathrm{W}^*$-bundles and isomorphism classes of principal $\Aut M$-bundles. As a consequence, using that $\Aut R$ is contractible, they prove that any locally trivial bundle with all fibres isomorphic to $R$ already has to be trivial \cite[Proposition 4.7, Theorem 4.10]{EvingtonPennig}.

However, Evington and Pennig leave open the question of whether or not every $\mathrm{W}^*$-bundle with all fibres isomorphic to some fixed factor $M$ is necessarily locally trivial. In this paper we resolve this question to the negative with the following counterexample over a space of covering dimension equal to zero.
\begin{alphathm}[See \Cref{cor:theexample} and \Cref{ex:nonloctrivial}]\label{thm:mainresult}
    The $\mathrm{W}^*$-bundle over $\mathbb{N}\cup\{\infty\}$ with all fibres isomorphic to the free group factor $L\mathbb{F}_\infty$ given by
    \begin{equation}\label{eq:thebundle}
        \N :=\{f\in C_\sigma(\mathbb{N}\cup\{\infty\}, M_2(L\mathbb{F}_\infty))\mid f(\infty)\in L\mathbb{F}_\infty\subseteq M_2(L\mathbb{F}_\infty)\},
    \end{equation}
    is not locally trivial.
\end{alphathm}

We prove \Cref{thm:mainresult} by introducing the following property for factorial $\mathrm{W}^*$-bundles, which can be thought of as a uniform version of fullness for factors.
\begin{alphadef}[{See \Cref{def:uniwspectralgap}}]
    A factorial $\mathrm{W}^*$-bundle $\M$ over $X$ has \emph{uniform w-spectral gap} if: for every $\varepsilon>0,$ there is a finite subset $ F$ of $\U(\M)$ and $\delta>0$ such that for $a\in (\M)_1$ we have
    \begin{equation}
        \max_{u\in F}\|[u,a]\|_{2,X}<\delta\implies \|a-E_{C(X)}(a)\|_{2,X}<\varepsilon.
    \end{equation}
\end{alphadef}
Subsequently we strengthen the analogy between full factors and factorial $\mathrm{W}^*$-bundles with uniform $w$-spectral gap, by giving several equivalent characterizations of the above property, including one in terms of central sequence algebras.

Thereafter, we prove that any locally trivial $\mathrm{W}^*$-bundle with non-Gamma fibres has uniform $w$-spectral gap (\Cref{prop:localityofspgap}). As a consequence, for factorial $\mathrm{W}^*$-bundles with non-Gamma fibres the negation of having uniform $w$-spectral gap forms an obstruction for being locally trivial. This obstruction is used to prove \Cref{thm:mainresult}. In fact we obtain the following more general result.
\begin{alphathm}[{cf. \Cref{thm:theonlyone}}]\label{thm:mainresult2}
    Let $N\subseteq M$ be finite factors with separable predual and consider the subtrivial bundle
    \begin{equation}
        \N =\{f\in C_\sigma(\mathbb{N}\cup \{\infty\}, M)\mid f(\infty)\in N\}.
    \end{equation}
    Then $\N$ has uniform $w$-spectral gap if and only if $N'\cap M = \C1$ and $N$ has $w$-spectral gap in $M$ in the sense of Popa \cite{PopaSpgap}.
\end{alphathm}

\emph{Outline.} The paper is structured as follows. \Cref{sec:Prelims} contains some preliminaries about $\mathrm{W}^*$-bundles, ultraproducts and inclusions of von Neumann algebras needed in the sequel. In \Cref{sec:spectralgap} we give the definition of having uniform $w$-spectral gap and prove some of its basic properties and consequences, especially that it is a local property. The final section, \Cref{sec:construction}, contains the main results of the paper. Here we give a complete characterization of when a subtrivial $\mathrm{W}^*$-bundle which has as fibres either a fixed non-Gamma factor $M$ or a fixed non-Gamma subfactor $N\subseteq M$ has uniform $w$-spectral gap, cf. \Cref{thm:theonlyone}. Finally we use this characterization to prove \Cref{thm:mainresult} and \Cref{thm:mainresult2}.
\begin{ack} 
    The author would like to thank Matteo Pagliero for useful comments on the manuscript and his supervisor, Jamie Gabe, for many fruitful discussions, useful suggestions regarding the paper, and especially for introducing the author to the theory of $\mathrm{W}^*$-bundles.
\end{ack}

\section{Preliminaries}\label{sec:Prelims}
\subsection{$\mathrm{W^*}$-bundles}
In this section we recall some basic results concerning tracially continuous $\mathrm{W}^*$-bundles (from now on called $\mathrm{W}^*$-bundles). As a main reference for $\mathrm{W}^*$-bundles we use \cite{Evington_thesis}.
\begin{definition}[{\cite[Section 5]{OzawaWbundle}}]
    \label{def:w*bundle}
    A \emph{$\mathrm{W}^*$-bundle over a compact Hausdorff space $X$} is a unital $\mathrm{C}^*$-algebra $\M$ equipped with a unital embedding of $C(X)$ into the centre of $\M$ and a faithful tracial\footnote{That is $E(ab) = E(ba)$ for every $a,b\in\M$.} conditional expectation $E: \M\to C(X)$ such that the norm-unit ball of $\M$ is complete with respect to the following \emph{uniform 2-norm}
    \begin{equation}
        \label{eq:uniform2norm}
        \|a\|_{2,X} := \|E(a^*a)\|^\frac{1}{2}.
    \end{equation}
    Two $\mathrm{W}^*$-bundles $\M_i$ over $X$ with expectation $E_i$ for $i=1,2$ respectively, are \emph{isomorphic} if there is a $*$-isomorphism $\alpha: \M_1\to \M_2$ such that $\alpha(C(X)) = C(X)$ and $E_2\circ \alpha = \alpha\circ E_1$.
\end{definition}
Given a $\mathrm{W}^*$-bundle $\M$ over $X$, each point $x\in X$ gives rise to a tracial state $\tau_x := \text{eval}_x\circ E$ on $\M$. Here $\text{eval}_x: C(X)\to \C$ is evaluation at $x$. Each such trace $\tau_x$ defines a 2-seminorm on $\M$ by the formula $\|a\|_{2,\tau_x} = \tau_x(a^*a)^\frac{1}{2}, a\in \M$. This relates to the norm defined in \eqref{eq:uniform2norm} as follows (and explains the name uniform 2-norm)
\begin{equation}\label{eq:unif2normspatial}
    \|a\|_{2,X} = \sup_{x\in X}\|a\|_{2,\tau_x}.
\end{equation}
For $x\in X$, denote by $\pi_{\tau_x}: \M\to B(L^2(\M,\tau_x))$ the GNS-representation with respect to $\tau_x$. Recall also that $\tau_x$ induces a normal faithful tracial state on $\pi_{\tau_x}(\M)''$, which by abuse of notation we also denote by $\tau_x$. As such, $(\pi_{\tau_x}(\M)'',\tau_x)$ is a tracial von Neumann algebra.

By \cite[Theorem 11]{OzawaWbundle} the co-restriction $\pi_{\tau_x} : \M\to \pi_{\tau_x}(\M)''$ is surjective (see also \cite[Theorem 3.2.9]{Evington_thesis} for the case where $X$ is not necessarily metrisable). As such the image $\pi_{\tau_x}(\M) =\pi_{\tau_x}(\M)''$ is already a tracial von Neumann algebra which we call the \emph{fibre of $\M$ at $x$} and denote by $\M_x$. Notice also that since $\tau_x$ is a tracial state, the $*$-homomorphism $\pi_{\tau_x}$ uniquely factors through a $*$-isomorphism
\begin{equation}
    \label{eq:abstractfibre}
    \M/I_x \cong \pi_{\tau_x}(\M)'',
\end{equation}
where $I_x = \{a\in \M\mid \tau_x(a^*a)=0\}$ is the kernel ideal of the tracial state $\tau_x$.
A $\mathrm{W}^*$-bundle is \emph{factorial} if all its fibres are factors.
\newline

Next, we introduce the two main examples of $\mathrm{W}^*$-bundles considered in this paper. The most basic examples are the so-called trivial $\mathrm{W}^*$-bundles, which were already mentioned in the introduction. Let $(M,\tau)$ be a tracial von Neumann algebra and $X$ a compact Hausdorff space. The \emph{trivial $\mathrm{W}^*$-bundle over $X$ with fibre $M$} is defined as
\begin{equation}
    \label{eq:trivwbundle}
    C_\sigma(X,M) = \{f: X\to M\mid  f \text{ is } \|\cdot\|\text{-bounded and } \|\cdot\|_{2,\tau}\text{-continuous}\}.
\end{equation}
Here the central embedding $\iota:C(X)\to C_\sigma(X,M)$ and expectation $E: C_\sigma(X,M)\to C(X)$ are given by
\begin{align}
     &\iota(f)(x) = f(x)1_M, &&f\in C(X), x\in X \text{  and,} \\&E(a)(x) = \tau(a(x)), &&a\in C_\sigma(X,M) , x\in X.
\end{align}

The second class of examples are \emph{subtrivial $\mathrm{W}^*$-bundles}. Given some trivial $\mathrm{W}^*$-bundle $C_\sigma(X,M)$ and a family $\{N_x\}_{x\in X}$ of von Neumann subalgebras of $M$. Then
\begin{equation}
    \N = \{f\in C_\sigma(X,M)\mid f(x)\in N_x, \forall x\in X\},
\end{equation}
is the \emph{subtrivial $\mathrm{W}^*$-bundle determined by $\{N_x\}_{x\in X}$}.
It is easy to see that indeed $\N$ is a $\mathrm{W}^*$-bundle over $X$ with respect to the embedding and expectation inherited from $C_\sigma(X,M)$, cf. \cite[Example 3.1.9]{Evington_thesis}.

Denote now by $\mathrm{eval}_x : \N\to M$ the $*$-homomorphism obtained by evaluating at $x\in X$. The isomorphism \eqref{eq:abstractfibre} implies that the fibre $\N_x$ is isomorphic to $\mathrm{eval}_x(\N)$.
It is immediate that $\mathrm{eval}_x(\N)\subseteq N_x$. The following result gives a characterization of when equality occurs, and hence of when the fibres of $\N$ are indeed given by the family $\{N_x\}_{x\in X}$.
\begin{proposition}[{\cite[Proposition 3.1.10]{Evington_thesis}}]
    \label{prop:subtrivialbundles}
    Let $\M = C_\sigma(X,M)$ be a trivial $\mathrm{W}^*$-bundle and $\N$ the subtrivial bundle defined by the family of von Neumann subalgebras $\{N_x\}_{x\in X}$. Then the following are equivalent:
    \begin{enumerate}[(i)]
        \item For all $x\in X$, $\mathrm{eval}_x(\N) = N_x$, (and hence $\N_x\cong N_x$).
        \item For all $b\in M$, the map $x\mapsto \mathrm{dist}_{\|\cdot\|_{2,\tau}}(b,N_x)$ is upper-semicontinuous, where $\mathrm{dist}_{\|\cdot\|_{2,\tau}}(b,N_x) = \inf\{\|b-c\|_{2,\tau}\mid c\in N_x\}$.
    \end{enumerate}
\end{proposition}
\begin{example}\label{ex:subtrivial}
    Suppose that $N\subseteq M$ is a von Neumann subalgera and $Y\subseteq X$ is a closed set. Then the family $\{N_x\}_{x\in X}$ defined by 
    \begin{equation}
        N_x = \begin{cases}
            N \quad\text{ if } x\in Y,\\
            M \quad \text{otherwise}.
        \end{cases}
    \end{equation}
    satisfies the conditions of \Cref{prop:subtrivialbundles}. The subtrivial bundles obtained in this way will play an important role in \Cref{sec:construction}.
\end{example}

The final background material on $\mathrm{W}^*$-bundles needed in this paper is the theory of restrictions and local triviality as introduced by Evington and Pennig in \cite{EvingtonPennig}. 
Let $\M$ be a $\mathrm{W}^*$-bundle over $X$ and $Y\subseteq X$ a closed subset. Evington and Pennig define \emph{$\M\lvert_Y$ the restriction of $\M$ to $Y$} to be the quotient $\M/I_Y$, where $I_Y = \{a\in\M\mid E(a^*a)(y) = 0, \text{ for all } y\in Y\}$, cf. \cite[Definition 2.8]{EvingtonPennig}. Furthermore, they prove in \cite[Proposition 2.9]{EvingtonPennig} that $\M\lvert_Y$ is a $\mathrm{W}^*$-bundle over $Y$ where the central embedding $C(Y)\to \M\lvert_Y$ and expectation $E\lvert_Y: \M\lvert_Y\to C(Y)$ are the unique  maps that make the following diagrams commutative.
\begin{equation}\label{eq:restrictiondiagrams1}
    \begin{tikzcd}
	\M && {\M\lvert_Y} \\
	\\
	{C(X)} && {C(Y)}
	\arrow[two heads, from=1-1, to=1-3]
	\arrow[hook, from=3-1, to=1-1]
	\arrow[two heads, from=3-1, to=3-3]
	\arrow[hook, from=3-3, to=1-3]
\end{tikzcd}
\end{equation}
and 
\begin{equation}
    \label{eq:restriciondiagrams2}
    \begin{tikzcd}
	\M && {\M\lvert_Y} \\
	\\
	{C(X)} && { C(Y)}
	\arrow[two heads, from=1-1, to=1-3]
	\arrow["E"', from=1-1, to=3-1]
	\arrow["{E\lvert_Y}"', from=1-3, to=3-3]
	\arrow[two heads, from=3-1, to=3-3]
\end{tikzcd}
\end{equation}
Here the bottom $*$-homomorphisms are given by restricting $f\in C(X)$ to $Y$.
Notice that from the above diagram \eqref{eq:restriciondiagrams2} it is immediate that for every $y\in Y$ also the following diagram commutes.
\begin{equation}
    \label{eq:tracediagram}
    \begin{tikzcd}
	\M && {\M\lvert_Y} \\
	{} & {} \\
	& \C
	\arrow[two heads, from=1-1, to=1-3]
	\arrow["{\mathrm{eval}_y\circ E}"', from=1-1, to=3-2]
	\arrow["{\mathrm{eval}_y\circ E\lvert_Y}", from=1-3, to=3-2]
\end{tikzcd}
\end{equation}
Hence by abuse of notation, we denote the traces that $y$ induces on $\M$ and $\M\lvert_Y$ both by $\tau_y$. This combined with \eqref{eq:abstractfibre}, shows that the fibres $\M_y$ and $(\M\lvert_Y)_y$ are canonically isomorphic.

Using restrictions of $\mathrm{W}^*$-bundles, one can define \emph{local properties} of $\mathrm{W^*}$-bundles. The following is fundamental in that sense.
\begin{definition}[{\cite[Definition 2.10]{EvingtonPennig}}]
    \label{def:locallytrivialw*bundle}
    A $\mathrm{W}^*$-bundle $\M$ over $X$ is \emph{locally trivial} if for every $x\in X$ there is closed neighbourhood $Y$ of $x$ such that $\M\lvert_Y$ is isomorphic to a trivial $\mathrm{W}^*$-bundle over $Y$ (as $\mathrm{W}^*$-bundles).
\end{definition}
\subsection{Tracial ultraproducts} In this section we introduce a notion of tracial ultraproduct that will be crucial in giving an algebraic characterization of the spectral gap property for $\mathrm{W}^*$-bundles that will be introduced in \Cref{sec:spectralgap}. The ultraproducts defined in this sections are particular instances of the tracial ultraproducts as considered in \cite[Section 5]{2024traciallycompletecalgebras}. However, for the sake of brevity, we only introduce them in the generality necessary for the subsequent results.
\begin{definition}\label{def:ultraproduct}
    Let $I$ be a non-empty directed set and $\omega$ a cofinal ultrafilter on $I$. Suppose that $\{\M_i\}_{i\in I}$ is a family of $\mathrm{W}^*$-bundles respectively over the compact Hausdorff spaces $\{X_i\}_{i\in I}$ and suppose that $\{Y_i\}_{i\in I}$ is a family of non-empty closed subsets $Y_i\subseteq X_i$. We define the \emph{tracial ultraproduct of $\{(\M_i,Y_i)\}_{i\in I}$} as 
    \begin{equation}
        \label{eq:tracialultraproduct}
        \prod_{i\in I}^\omega (\M_i,Y_i):= \prod_{i\in I}\M_i/\{(a_i)_{i\in I} \mid \lim_{i\to \omega}\|a_i\|_{2,Y_i} = 0\}.
    \end{equation}
    Here $\|a_i\|_{2,Y_i} = \sup_{y\in Y_i}\|a_i\|_{2,\tau_y}$ or equivalently it is the uniform $2$-norm in the restriction $\M_i\lvert_{Y_i}$.
\end{definition}
By combining \cite[Proposition 3.9]{2-coloured} with \cite[Proposition 5.6]{2024traciallycompletecalgebras}\footnote{The results in \cite{2-coloured,2024traciallycompletecalgebras} are only stated when the directed set $I$ equals the natural numbers. However their proofs generalize immediately to arbitrary directed sets.} and \cite[Proposition 2.9]{EvingtonPennig} the tracial ultraproduct $\prod_i^\omega(\M_i,Y_i)$ as above is canonically a $\mathrm{W}^*$-bundle over the ultrasum $\sum_\omega Y_i$\footnote{Here the ultrasum $\sum_\omega Y_i$ is the Gelfand spectrum of the norm-ultraproduct $\prod_\omega C(Y_i)$.}.

In particular, when the $\{\M_i\}_{i\in I}$ are $\mathrm{W}^*$-bundles over points $\{x_i\}$, i.e. tracial von Neumann algebras, then the ultraproduct $\prod_i^\omega (\M_i,x_i)$ is again a $\mathrm{W}^*$-bundle over a point, that is again a tracial von Neumann algebra. Furthermore, from the definition \eqref{eq:tracialultraproduct} it is clear that in this case $\prod_i^\omega (\M_i, x_i)$ agrees with $\prod_i^\omega (\M_i,\tau_{x_i})$ the usual ultraproduct of tracial von Neumann algebras.  

We will mostly be considering tracial ultraproducts where the $\mathrm{W}^*$-bundles are all equal to some fixed bundle $\M$ over a fixed space $X$, and the closed subsets $\{Y_i\}_i$ are either all equal to $X$ or all singletons. In the former case one obtains an obvious diagonal inclusion $\M\subseteq (\M,X)^\omega$. 
In the latter case one also obtains a \emph{diagonal $*$-homomorphism} $\iota^\omega :\M\to \prod^\omega(\M,x_i)$ by composing the diagonal embedding $\M\to \prod_{i\in I} \M$ with the quotient map. In this situation we also have the following further relation with ultraproducts of tracial von Neumann algebras.
\begin{proposition}\label{prop:comparisonofultraprods}
    Let $\M$ be a $\mathrm{W}^*$-bundle over $X,$ $(x_i)_{i\in I}$ a net in $X$ and $\omega$ a cofinal ultrafilter on $I$. Then the map
    \begin{align*}
        \Phi : \prod_{i\in I}^\omega (\M,x_i) &\to \prod_{i\in I}^\omega (\pi_{\tau_{x_i}}(\M)'',\tau_{x_i}):\\
        (a_i)^\omega&\mapsto (\pi_{\tau_{x_i}}(a_i))^\omega,
    \end{align*}
    defines a $*$-isomorphism making the following diagram commutative.
    \begin{equation}\label{eq:diagcomparison}
        \begin{tikzcd}
	&& {\prod_{i\in I}^\omega(\M,x_i)} \\
	\M && {\prod_{i\in I}^\omega(\pi_{\tau_{x_i}}(\M)'',{\tau_{x_i}})} \\
	&& {}
	\arrow["\Phi", "\cong"', from=1-3, to=2-3]
	\arrow["{\iota^\omega}", from=2-1, to=1-3]
	\arrow["{(\pi_{\tau_{x_i}})^\omega}"', from=2-1, to=2-3]
\end{tikzcd}
    \end{equation}
\end{proposition}
\begin{proof}
   We show that $\Phi$ is a $*$-isomorphism by showing that the kernel of the following surjective $*$-homomorphism
   \begin{equation}
       \tilde{\Phi}: \prod_{i\in I} \M \to \prod_{i\in I}^\omega(\pi_{\tau_{x_i}}(\M)'',{\tau_{x_i}}): (a_i)_i\mapsto (\pi_{\tau_{x_i}}(a_i))_i,
   \end{equation}
   is exactly $\{(a_i)_i\in \prod_{i\in I} \M \mid\lim_{i\to\omega}\|a_i\|_{2,\tau_{x_i}}=0\}$. However, this follows directly from the fact that for each $i\in I$ we have $\|a_i\|_{2,\tau_{x_i}}= \|\pi_{\tau_{x_i}}(a_i)\|_{2,\tau_{x_i}}$.

   Commutativity of the diagram \eqref{eq:diagcomparison} is now trivial.
\end{proof}

\subsection{Spectral gap for inclusions of von Neumann algebras}
Recall that $\mathrm{II}_1$ factor $(M,\tau)$ is called \emph{full} or \emph{non-Gamma} if every norm-bounded net $(a_i)_{i\in I}$ in $M$ that is central, i.e. such that $\lim_i\|[b,a_i]\|_{2,\tau} = 0$ for every $b\in M$, is also asymptotically trivial that is $\lim_i\|a_i-\tau(a_i)1_M\|_{2,\tau} = 0$. Due to a very remarkable result of Connes this is equivalent to the conjugation action $\U(M)\curvearrowright L^2(M,\tau)$ having a spectral gap property \cite{Connesclass} (see also \cite{marrackchigap} for the non-separable case).

In \cite{PopaSpgap}, Popa generalizes the notions of  non-Gamma factors and factors with the above mentioned spectral gap property to the situation of a von Neumann subalgebra $N$ in some $\mathrm{II}_1$ factor $(M,\tau)$. Popa calls the two notions $N$ has $w$-spectral gap in $M$ and $N$ has spectral gap in $M$ respectively. However, unlike the case for a single factor, the two notions need not agree, see for example the appendix of \cite{IoanaVaes}.

In this paper we only make use of the $w$-spectral gap property for inclusions of von Neumann algebras. The difference between the stronger (actual) spectral gap property is quite subtle, the only alteration in the definition below is that $(M)_1$ would have to be replaced with $M$. However, since we will only utilize the $w$-spectral gap property, we refer the reader to Popa's article for details concerning the stronger version.
\begin{definition}[{\cite[Remark 2.2]{PopaSpgap}}]
    \label{def:spgapinclusions}
    Let $(M,\tau)$ be a $\mathrm{II}_1$ factor and $N\subseteq M$ a von Neumann subalgebra. We say \emph{$N$ has $w$-spectral gap in $M$} if: $\forall\varepsilon>0,\exists F\subseteq \U(N)$ finite and $\delta>0$ such that if $a\in (M)_1$ satisfies $\max_{u\in F}\|[u,a]\|_{2,\tau}<\delta$, then $\|a- E_{N'\cap M}(a)\|_{2,\tau}<\varepsilon$. Here $E_{N'\cap M}: M\to N'\cap M$ is the unique trace-preserving conditional expectation onto $N'\cap M$.
\end{definition}
\begin{remark}
    Notice that a $\mathrm{II}_1$ factor $(M,\tau)$ is full if and only if $M$ has $w$-spectral gap inside itself.
\end{remark}
The following proposition gives a convenient algebraic characterization of $w$-spectral gap for inclusions. The argument is routine, however for convenience of the reader a proof is included.
\begin{proposition}    \label{prop:wspecgapincl} Let $(M,\tau)$ be a $\mathrm{II}_1$ factor and $N$ a von Neumann subalgebra of $M$. Then $N$ has $w$-spectral gap in $M$ if and only if $N'\cap M^\omega = (N'\cap M)^\omega$ for every directed set $I\neq \emptyset$ and cofinal ultrafilter $\omega$ on $I$.
\end{proposition}
\begin{proof}
    Suppose that $N$ does not have $w$-spectral gap in $M$. Then $\exists\varepsilon>0$ such that for every finite $F\subseteq\U(N)$ and $\delta>0$ there exists an $a_{F,\delta}\in (M)_1$ satisfying $\max_{u\in F}\|[u,a_{F,\delta}]\|_{2,\tau}<\delta$ and $\|a_{F,\delta}-E_{N'\cap M}(a_{F,\delta})\|_{2,\tau}>\varepsilon.$
    Define the directed set
    \begin{equation}
        I = \{(F,\delta) \in \mathcal{U}(N)\times \R_{>0}\mid F \text{ is finite}\},
    \end{equation}
    ordered by
    \begin{equation}\label{eq:ordering}
        (F,\delta)\leq (F',\delta')\iff F\subseteq F' \text{ and } \delta\geq\delta'.
    \end{equation}
    Let $\omega$ be some cofinal ultrafilter on $I$. Since $\|a_{F,\delta}\|\leq 1, \forall (F,\delta)\in I,$ we have $(a_{F,\delta})_{F,\delta}\in M^\omega$. Since $\max_{u\in F}\|[u,a_{F,\delta}]\|_{2,\tau}<\delta,\forall (F,\delta)\in I$ and $\omega$ is cofinal, we have $(a_{F,\delta})_{F,\delta}\in N'\cap M^\omega$. However, $(a_{F,\delta})_{F,\delta}\notin (N'\cap M)^\omega$ because $\|a_{F,\delta}-E_{N'\cap M}(a_{F,\delta})\|_{2,\tau}>\varepsilon, \forall (F,\delta)\in I$.

    Conversely suppose $N$ does have $w$-spectral gap inside $M$ and let $I\neq \emptyset$ be a directed set and $\omega$ a cofinal ultrafilter on $I$. Suppose $(a
    _i)_{i\in I}\in N'\cap M^\omega$, we have to show that $(a_i)_{i\in I}\in (N'\cap M)^\omega$ or equivalently
    \begin{equation}
        \lim_{i\to\omega}\|a_i-E_{N'\cap M}(a_i)\|_{2,\tau}=0.
    \end{equation}
    By possibly rescaling we may assume that each $\|a_i\|\leq 1$.
    Now let $\varepsilon>0$, then by definition there is a finite $F \subseteq\mathcal{U}(N)$ and $\delta>0$ such that if $a\in (M)_1$ satisfies $\max_{u\in F}\|[u,a]\|_{2,\tau}<\delta$, then also $\|a-E_{N'\cap M}(a)\|_{2,\tau}<\varepsilon$. Since $(a_i)\in N'\cap M^\omega$, we have $\lim_{i\to\omega}\|[b,a_i]\|_{2,\tau} = 0$ for every $b\in N$. Thus since $F\subseteq N$ is finite, one sees that
    \begin{equation}
        J =\{i\in I\mid \max_{u\in F}\|[u,a_i]\|_{2,\tau}<\delta\},
    \end{equation}
    is an element of the ultrafilter $\omega$. By the defining property of $F$ and $\delta$, and since every $ \|a_i \|\leq 1$, each $i\in J$ also satisfies $\|a_i-E_{N'\cap M}(a_i)\|_{2,\tau}<\varepsilon$. Hence also
    \begin{equation}
        \{i\in I\mid \|a_i-E_{N'\cap M}(a_i)\|_{2,\tau}<\varepsilon\}
    \end{equation}
    is an element of the ultrafilter $\omega$. This proves $\lim_{i\to \omega}\|a_i-E_{N'\cap M}(a_i)\|_{2,\tau} =0$, since $\varepsilon>0$ was arbitrary. 
\end{proof}
When $M$ has separable predual (equivalently $M$ is $\|\cdot\|_{2,\tau}$-separable) then an even simpler algebraic characterization can be given using a single free ultrafilter on $\mathbb{N}$.
\begin{proposition}[{\cite[Remark 2.2]{PopaSpgap}}]
    \label{prop:sepultrafilber}
    Let $(M,\tau)$ be a $\mathrm{II}_1$ factor with separable predual and $N$ a von Neumann subalgebra of $M$. Then $N$ has $w$-spectral gap inside $M$ if and only if $N'\cap M^\omega = (N'\cap M)^\omega$ for some (or any) free ultrafilter on $\mathbb{N}$.
\end{proposition}

\section{A notion of spectral gap for $\mathrm{W^*}$-bundles}\label{sec:spectralgap}
 Following the naming scheme set out by Popa, we define a notion of $w$-spectral gap for factorial $\mathrm{W}^*$-bundles.
 \begin{definition}
    \label{def:uniwspectralgap} A factorial $\mathrm{W}^*$-bundle $\M$ over $X$ has \emph{uniform $w$-spectral gap} if: $\forall\varepsilon>0,$ there exists a finite subset $F$ of $\U(\M)$ (or of $\M$) and $\delta>0$ such that if $a\in (\M)_1$ satisfies $\max_{u\in F}\|[u,a]\|_{2,X}<\delta$, then $\|a-E(a)\|_{2,X}<\varepsilon$.
\end{definition}
\begin{remark}
    In the definition it is equivalent to consider finite subsets of either $\U(\M)$ or $\M$, since every element of $\M$ is a linear combination of at most four unitaries.
\end{remark}
\begin{example}\label{ex:trivialfull}
    Whenever $(M,\tau)$ is a full factor, then the trivial bundle $C_\sigma(X,M)$ has uniform $w$-spectral gap. Indeed because $M$ is full, for every $\varepsilon>0$ there is a finite $F_0 = \{u_1,\dots,u_n\}\subseteq \U(M)$ and $\delta>0$ such that if $a\in (M)_1$ satisfies $\max_{u\in F_0}\|[u,a]\|_{2,\tau}<\delta$ then $\|a-\tau(a)1_M\|_{2,\tau}<\varepsilon$. Now defining $F = \{U_1,\dots ,U_n\}\subseteq \U(C_\sigma(X,M))$ where $U_i(x) = u_i, x\in X$ does the trick by \eqref{eq:unif2normspatial}.
\end{example}

In the same spirit as for inclusions of von Neumann algebras, using ultraproducts one can obtain an algebraic characterization of when a $\mathrm{W}^*$-bundle has uniform $w$-spectral gap. The following proposition also demonstrates why having uniform $w$-spectral gap can be viewed as a generalization of fullness of factors to the situation of $\mathrm{W}^*$-bundles. Indeed since a $\mathrm{W}^*$-bundle $\M$ over $X$ is always equipped with a central embedding $C(X)\subseteq \M$, the following results states that a factorial $\mathrm{W}^*$-bundle has uniform $w$-spectral gap if and only if its central sequence algebras (in the tracial sense) are as trivial as can be.
\begin{proposition}
    \label{prop:trivcentrseqalgs.}
    Let $\M$ be a factorial $\mathrm{W}^*$-bundle over $X$. Then $\M$ has uniform $w$-spectral gap if and only if $\M'\cap(\M,X)^\omega = C(X)_\omega$ for any directed set $I\neq \emptyset$ and cofinal ultrafilter on $\omega$ on $I$.
    Furthermore, if $\mathrm{\M}$ is $\|\cdot\|_{2,X}$-separable it suffices to consider a single free ultrafilter on $\mathbb N$.
\end{proposition}
\begin{proof}
    First notice that since $E : \M\to C(X)$ is a conditional expectation, it will especially satisfy $E(a^*)E(a)\leq E(a^*a)$. Using this property and the definition of the $\|\cdot\|_{2,X}$-norm one can easily see that for any directed set $I$ and cofinal ultrafilter $\omega$ on $I$, one has for every $(a_i)_i\in (\M,X)^\omega:$
    \begin{equation}
        \label{eq:centralalg}
        (a_i)_i\in C(X)_\omega\subseteq (\M,X)^\omega\iff \lim_{i\to\omega}\|a_i-E(a_i)\|_{2,X}=0.
    \end{equation}
        
    Suppose that $\M$ does not have uniform $w$-spectral gap. Then using an argument almost identical to the one in the proof of \Cref{prop:wspecgapincl}, one can find an $\varepsilon>0$ and construct a net of contractions $(a_{F,\delta})_{F,\delta}$ in $\M$, indexed over the directed set $I$ consisting of pairs $(F,\delta)$ where the $F$ are finite sets of unitaries in $\M$ and the $\delta$ strictly positive real numbers and ordered as in \eqref{eq:ordering}, satisfying 
    \begin{equation}\label{eq:netconstruct}
        \max_{u\in F}\|[u,a_{F,\delta}]\|_{2,X}<\delta \text{ and } \|a_{F,\delta}-E(a_{F,\delta})\|_{2,X}>\varepsilon,\qquad \forall (F,\delta)\in I.
    \end{equation} 
    Thus, for any cofinal ultrafilter $\omega$ on $I$, the inequalities in \eqref{eq:netconstruct} in combination with \eqref{eq:centralalg} ensure that $(a_{F,\delta})_{F,\delta}\in \M'\cap (\M,X)^\omega\setminus C(X)_\omega$.

    Conversely assume that $\M$ does have uniform $w$-spectral gap and let $I\neq \emptyset$ be a directed set and $\omega$ a cofinal ultrafilter on $I$. Let $(a_i)_i\in \M'\cap (\M,X)^\omega$ we have to show that $(a_i)_i\in C(X)_\omega \subseteq (\M, X)^\omega$. By \eqref{eq:centralalg} this is equivalent to showing that
    \begin{equation}
        \lim_{i\to \omega}\| a_i-E(a_i)\|_{2,X} = 0.
    \end{equation}
    However, this follows from a very similar argument as in the proof of \Cref{prop:wspecgapincl}, using basic properties of ultrafilters and the definition of having uniform $w$-spectral gap.

    For the \say{furthermore part}, assume that $\M$ is $\|\cdot\|_{2,X}$-separable and consider a free ultrafilter $\omega$ on $\mathbb N$. The proof is complete if we can show that $\M'\cap (\M,X)^\omega = C(X)_\omega$ implies that $\M$ has uniform $w$-spectral gap. So suppose by contradiction that $\M$ does not have uniform $w$-spectral gap. As in the beginning of this proof, we can find an $\varepsilon>0$ and a net of contractions $(a_{F,\delta})_{F,\delta}$ in $\M$ indexed over the directed set $I$ as above, satisfying \eqref{eq:netconstruct}. Now take $(u_n)_{n\geq 1}$ a $\|\cdot\|_{2,X}$-dense sequence in $\U(\M)$. Define the subnet
    \begin{equation}
        a_n := a_{F_n,\frac{1}{n}}\in (\M)_1, \qquad \text{ for } F_n =\{u_1,\dots, u_n\}.
    \end{equation}
    Then since the $\|\cdot\|_{2,X}$-norm satisfies
    \begin{equation}
        \|ab\|_{2,X}\leq \min\{\|a\|\|b\|_{2,X}, \|a\|_{2,X}\|b\|\},\qquad \forall a,b\in \M,
    \end{equation}
    $\|\cdot\|_{2,X}$-density of $(u_n)_{n\geq 1}$ in $\U(\M)$ implies $(a_n)_n\in \M'\cap (\M,X)^\omega$. However, since $\|a_n-E(a_n)\|_{2,X}>\varepsilon, \forall n\geq1$ we have $(a_n)_n\notin C(X)_\omega\subseteq (\M,X)^\omega$ by \eqref{eq:centralalg}, reaching a contradiction.
\end{proof}
\Cref{prop:closedgroups} is also analogous to the situation for $\rm{II}_1$ factors, as is its proof cf. \cite[Theorem 15.3.2]{anantharaman-2010}; which makes use of the following lemma.
\begin{lemma}
    \label{lemma:closedgroups}
    Let $\M$ be a factorial $\mathrm{W}^*$-bundle over $X$, which has uniform $w$-spectral gap. Then for every $\varepsilon>0$ there is a finite $F\subseteq \M$ and $\delta>0$ such that if $u\in \U(\M)$ satisfies $\max_{a\in F}\|[u,a]\|_{2,X}<\delta$ then
    \begin{equation}
        \inf\{\|u-g\|_{2,X}\mid g\in C(X,\T) = \U(C(X))\}<\varepsilon.
    \end{equation}
\end{lemma}
\begin{proof}
     Assume by contradiction that there is an $\varepsilon>0$ such that for every $\delta>0$ and finite $F\subseteq \M$ there is a unitary $u_{F,\delta}\in \U(\M)$ satisfying $\max_{a\in F}\|[u_{F,\delta},a]\|_{2,X}<\delta$ and
    \begin{equation}\label{eq:distgreater}
        \inf_{g\in C(X,\T)}\|u_{F,\delta}-g\|_{2,X}\geq\varepsilon.
    \end{equation}
    Let now $I$ be the set consisting of pairs $(F,\delta)$ of finite subsets of $\M$ and positive numbers ordered as in \eqref{eq:ordering}. Let $\omega$ be some cofinal ultrafilter on $I$, then $(u_{F,\delta})_{F,\delta}\in \M'\cap (\M,X)^\omega$.
    Since $\M$ has uniform $w$-spectral gap, \Cref{prop:trivcentrseqalgs.} implies $\M'\cap (\M,X)^\omega = C(X)_\omega$. Hence $(u_{F,\delta})_{F,\delta}$ represents a unitary in $C(X)_\omega$. Thus, since $C(X)_\omega$ is a norm-ultrapower, we can find a net of unitaries $g_{F,\delta}\in C(X,\T)$ such that $(u_{F,\delta})_{F,\delta} = (g_{F,\delta})_{F,\delta}$ in $(\M,X)^\omega$. Equivalently $\lim_{(F,\delta)\to \omega}\|u_{F,\delta}-g_{F,\delta}\|_{2,X}=0$, this contradicts \eqref{eq:distgreater}.
\end{proof}
\begin{proposition}
    \label{prop:closedgroups}
    Let $\M$ be a $\|\cdot\|_{2,X}$-separable factorial $\mathrm{W}^*$-bundle over $X$, then $\Aut \M$ is a Polish group with respect to the point-$\|\cdot\|_{2,X}$ topology. If furthermore $\M$ has uniform $w$-spectral gap, then $\Inn \M\subseteq \Aut \M$ is closed in said topology.
\end{proposition}
\begin{proof}
    Since $\M$ is factorial we really have $C(X) = Z(\M)$. Consequently by \cite[Proposition 3.10]{2024traciallycompletecalgebras}, $\Aut \M$ is a subgroup of $\Iso (\M, \|\cdot\|_{2,X})$ the isometry group of $(\M,\|\cdot\|_{2,X})$. Because $\M$ is $\|\cdot\|_{2,X}$-separable, the latter is a Polish group in the point-$\|\cdot\|_{2,X}$ topology (see \cite[Chapter I.9]{kechris}). Hence so is the former since $\Aut\M\leq \Iso(\M,\|\cdot\|_{2,X})$ is closed, due to $\|\cdot\|_{2,X}$-continuity of multiplication on bounded sets and $\|\cdot\|_{2,X}$-continuity of the adjoint.

    Moreover, since the unit ball of $\M$ is complete with respect to $\|\cdot\|_{2,X}$ and $\M$ is $\|\cdot\|_{2,X}$-separable, also $\U(\M)$ is a Polish group when equipped with the topology induced by the complete metric 
    \begin{equation}
        \label{eq:metricUM}
         d(u,v) = \|u-v\|_{2,X},\qquad u,v\in \U(\M).
    \end{equation}

    Assume now that $\M$ has uniform $w$-spectral gap, then for each $n\geq 1$, \Cref{lemma:techincallemma} supplies a $\delta_n>0$ and finite subset $F\subseteq \M$ such that if $u\in \U(\M)$ satisfies $\max_{a\in F}\|[u,a]\|_{2,X}<\delta_n$ then 
    \begin{equation}
        d(u, C(X,\T)) = \inf_{g\in C(X,\T)}\|u-g\|_{2,X}<2^{-n}.
    \end{equation}
    Let $\alpha$ be in the closure of $\Inn \M$, then there is a sequence of unitaries $v_n\in \U(\M)$ such that $\Ad v_n\to \alpha$. Since $\Aut \M$ is a Polish group, inversion of elements is continuous. Thus after possibly passing to a subsequence, we may assume that for each $n$ we have
    \begin{equation}
        \max_{a\in F_n}\|[v_{n+1}^{-1}v_n,a]\|_{2,X} =  \max\|((\Ad v_{n+1})^{-1}\circ \Ad v_n)(a)-a\|_{2,X}<\delta_n.
    \end{equation}
    Consequently $d(v_{n+1}^{-1}v_n, C(X,\T))<2^{-n}$. Since $C(X,\T)\subseteq Z(\M)$ we can thus inductively define unitaries $u_n\in \U(\M)$ such that $d(u_{n+1},u_n) = \|u_{n+1}-u_n\|_{2,X}<2^{-n}$ and $\Ad u_n = \Ad v_n$. Completeness of $\U(\M)$ then implies there is some unitary $u\in \U(\M)$ such that $u_n\to u$ w.r.t. $\|\cdot\|_{2,X}$ . Finally, this implies $\Ad u_n\to \Ad u$, hence $\alpha = \Ad u$. 
\end{proof}
\begin{remark}
    \Cref{prop:closedgroups} also has a partial converse. Indeed, if one assumes in addition to $\Inn \M\subseteq \Aut \M$ being closed that also every unitary in $\M'\cap (\M,X)^\omega$ lifts to a sequence of unitaries in $\M$. Then the proof of $(i)\implies (ii)$ of \cite[Theorem 15.3.2]{anantharaman-2010} immediately generalizes to the situation of $\mathrm{W}^*$-bundles, to show that $\M$ has uniform $w$-spectral gap.

    Notice furthermore, that whenever $\M$ has uniform $w$-spectral gap, then this lifting-of-unitaries properties is definitely satisfied. Indeed, in this case $\M'\cap (\M,X)^\omega = C(X)_\omega$, so every unitary in $\M'\cap (\M,X)^\omega$ even lifts to a sequence of unitaries in $C(X)$ (this property was crucially exploited in the proof of \Cref{lemma:closedgroups}).

\end{remark}

The following proposition gives several a priori stronger, local characterizations of the uniform $w$-spectral gap property. This will be the crucial ingredient for proving that the uniform $w$-spectral gap property is a local property.
\begin{proposition}
    \label{prop:equivchars}
    Let $\mathcal{M}$ be a factorial $\mathrm{W}^*$-bundle over $X$. Then the following are equivalent.
    \begin{enumerate}[(i)]
        \item $\M$ has uniform $w$-spectral gap.
        \item For every $\varepsilon>0,$ there exists a finite subset $F$ of $\U(\M)$ (or of $\M$) and $\delta>0$ such that if $a\in (\M)_1$ and $x\in X$ satisfy $\max_{u\in F}\|[u,a]\|_{2,\tau_x}<\delta$, then $\|a-\tau_x(a)1_\M\|_{2,\tau_x }<\varepsilon$.
        \item If $(a_i,x_i)_{i\in I}$ is a net in $(\M)_1\times X$ such that $\lim_i\|[b,a_i]\|_{2,\tau_{x_i}} = 0$ for every $b\in\M$, then $\lim_i\|a_i-\tau_{x_i}(a_i)1_\M\|_{2,\tau_{x_i}} = 0$
        \item For any net $(x_i)_{i\in I}$ in $X$ and cofinal ultrafilter $\omega$ on $I$ we have
            \begin{equation*}
                \iota^\omega(\M)'\cap \prod_{i\in I}^\omega(\M,x_i) = \C1.
            \end{equation*}
        \item For any net $(x_i)_{i\in I}$ in $X$ and cofinal ultrafilter $\omega$ on $I$ we have
        \begin{equation*}
                (\pi_{\tau_{x_i}})^\omega(\M)'\cap \prod_{i\in I}^\omega(\pi_{\tau_{x_i}}(\M)'',\tau_{x_i}) = \C1.
            \end{equation*}
    \end{enumerate}
\end{proposition}
\begin{proof}
    $(i)\implies (ii):$ Let $\varepsilon>0$ and pick $F,\delta$ as in \Cref{def:uniwspectralgap}. We claim that $F,\delta$ also suffice for $(ii)$. So suppose $a\in (\M)_1$ and $x\in X$ satisfy $\max_{u\in F}\|[u,a]\|_{2,\tau_x}< \delta$. As $y\mapsto \max_{u\in F}\|[u,a]\|_{2,\tau_y}$ is continuous, there is an open neighbourhood $U\subseteq X$ of $x$ such that $\max_{u\in F}\|[u,a]\|_{2,\tau_y}<\delta,\forall y \in U$. Let now $f\in C(X,[0,1])$ be a function, such that $f(x)=1$ and $f(y)=0, \forall y\notin U$. Then $fa\in (\M)_1$ satisfies $\max_{u \in F} \|[u,fa]\|_{2,X}<\delta$ and so $\|fa-E(fa)\|_{2,X}< \varepsilon$. In particular, since $f(x)=1$, we find $\|a-\tau_x(a)1_\M\|_{2,\tau_x}<\varepsilon$. \\
    $(ii)\implies (i):$ Obvious from \eqref{eq:unif2normspatial}.\\
    $(ii)\implies (iii):$ This is immediate.\\
    $(iii)\implies (ii):$ This follows from a, by now, routine argument. Indeed, assume by contradiction that $(ii)$ is not true. Then there is an $\varepsilon>0$ such that for any finite subset $F\subseteq\M$ and $\delta>0$ there are $a_{F,\delta}\in (\M)_1$ and $x_{F,\delta}\in X$ that satisfy $\max_{b\in F}\|[b,a_{F,\delta}]\|_{2,\tau_{x_{F,\delta}}}<\delta$ but $\|a_{F,\delta}-\tau_{x_{F,\delta}}(a_{F,\delta})1_\M \|_{2,\tau_{x_{F,\delta}}}>\varepsilon$. Then the net $(a_{F,\delta}, x_{F,\delta})$ where the $(F,\delta)$ are ordered by
    \begin{equation}
        (F,\delta)\leq (F',\delta')\iff F\subseteq F' \text{ and } \delta\geq\delta',
    \end{equation}
    clearly contradicts $(iii)$.\\
    $(iii)\implies (iv):$ This is trivial.\\
    $(iv)\implies (iii):$ By \Cref{prop:comparisonofultraprods}, $\prod^\omega (\M, x_i)$ is a finite factor with trace $\tau^\omega$ given by
    \begin{equation}
        \tau^\omega((a_i)_i) = \lim_{i\to\omega}\tau_{x_i}(a_i),\qquad \text{ for }(a_i)_i\in \prod_{i\in I}\M.
    \end{equation}
    From this, the contraposition of $(iv)\implies (iii)$ immediately follows.\\
    The equivalence $(iv)\iff(v)$ follows from \Cref{prop:comparisonofultraprods}.\\
    This proves all the equivalences. 
    
\end{proof}
\begin{corollary}\label{cor:fullfibres}
    Let $\M$ be a factorial $\mathrm{W}^*$-bundle that has uniform $w$-spectral gap. Then the fibres of $\M$ are full factors.
\end{corollary}
\begin{proof}
    This is immediate from surjectivity of the maps $\pi_{\tau_x}: \M\to \pi_{\tau_x}(\M)''$ and \Cref{prop:equivchars}$(iii)$, by considering constant nets $(x_i)_i = (x)_i$.
\end{proof}
\Cref{prop:equivchars} also gives rise to a second class of examples of $\mathrm{W}^*$-bundles with uniform $w$-spectral gap: tracial completions of universal group $\mathrm{C}^*$-algebras of groups with property (T). This is explained in the following example.
\begin{example}
    \label{ex:propTgroups}
    By \cite[Theorem 12.1.7]{brownozawa}) a countable discrete group $\Gamma$ has \emph{Kazhdan's property \rm{(T)}} \cite{kazhdan}, if and only if there exist a finite subset $F\subseteq \Gamma$ and $c>0$ with the following property: if $\rho : \Gamma\to B(\H)$ is a unitary representation of $\Gamma$ and $P_\rho$ is the orthogonal projection of $\H$ onto the subspace of all $\Gamma$-invariant vectors, then for every $\xi\in \H$ one has
    \begin{equation}\label{eq:strongT}
        \max_{g\in F}\|\rho(g)\xi-\xi\|_2\geq c \|\xi-P_\rho \xi\|_2.
    \end{equation}

    Fix now a countable discrete group $\Gamma$ with property (T) and a pair $(F,c)$ as above. Consider the universal group $\mathrm{C}^*$-algebra $A =C^*(\Gamma) = C^*(u_g\mid g\in \Gamma)$. Since $\Gamma$ has property $\mathrm{(T)}$, \cite[Corollary 1.6]{levit2024spectralgapcharacterlimits} implies its simplex of tracial states $T(A)$ is a Bauer simplex, that is the extreme boundary $\partial_e T(A)$ is compact (in the weak*-topology). Hence by \cite[Theorem 3.37]{2024traciallycompletecalgebras} (see also \cite{OzawaWbundle}) the \emph{tracial completion}
    \begin{equation}
        \label{eq:tracialcompletion}
        \M := \overline{A}^{T(A)}= \frac{\{(a_n)_{n\geq 1}\in \ell^\infty(A)\mid (a_n)_{n\geq 1} \text{ is $\|\cdot\|_{2,T(A)}$-Cauchy}\}}{\{(a_n)_{n\geq 1}\in \ell^\infty(A)\mid (a_n)_{n\geq 1} \text{ is $\|\cdot\|_{2,T(A)}$-null}\}},
    \end{equation}
    is canonically a $\mathrm{W}^*$-bundle over $X:= \partial_eT(A)$.\footnote{Here, the seminorm $\|\cdot\|_{2,T(A)}$ is defined as $\|a\|_{2, T(A)} = \sup_{\tau\in T(A)}\|a\|_{2,\tau} =\sup_{\tau\in \partial_eT(A)}\|a\|_{2,\tau} $, see \cite[Proposition 3.3]{2024traciallycompletecalgebras}.} Furthermore, the above construction comes equipped with a unital diagonal $*$-homomorphism $\alpha : A\to \M$ with $\|\cdot\|_{2,X}$-dense image. The $*$-homomorphism $\alpha$ also has the following property. If $x\in X = \partial_eT(A)$ is an extremal trace  of $A$ then the tracial state $\tau_x$ on $\M$ extends $x$ (that is $x = \tau_x\circ\alpha$), furthermore if $\pi_x : A\to \pi_x(A)''$ and $\pi_{\tau_x}:\M\to \pi_{\tau_x}(\M)''$ are the respective GNS-representations, then there is a unique $*$-isomorphism $\theta_x : \pi_x(A)''\to\pi_{\tau_x}(\M)''$ that makes the diagram 
    \begin{equation}
        \label{eq:identificationcompletion}
        \begin{tikzcd}
	A && {\pi_x(A)''} \\
	\\
	\M && {\pi_{\tau_x}(\M)''}
	\arrow["{\pi_x}", from=1-1, to=1-3]
	\arrow["\alpha"', from=1-1, to=3-1]
	\arrow["{\theta_x}", "\cong"', from=1-3, to=3-3]
	\arrow["{\pi_{\tau_x}}", two heads, from=3-1, to=3-3]
\end{tikzcd}
    \end{equation}
    commute, see \cite[Proposition 3.23]{2024traciallycompletecalgebras}.

    Finally, let $x\in X$ be an extremal trace on $A$, then $\pi_x(A)''$ is a finite factor generated by the $\pi_x(u_g), g\in \Gamma$. We also denote its unique trace by $\tau_x$ (the diagram \eqref{eq:identificationcompletion} justifies this notation). Consider now the unitary representation $\rho_x: \Gamma \to B(L^2(A,x))$ given by unitary conjugation, that is $(\rho_{x}(g))(\xi) = \pi_x(u_g)\pi_x^\text{op}(u_g^*)\xi$\footnote{Here the opposite representation $\pi_x^\text{op}$ is well-defined since $x$ is a tracial state.}. 
    Since the unitaries $\pi_x(u_g), g\in \Gamma$ generate the finite factor $\pi_x(A)''$, which is in standard form on $L^2(A,x)$, the space of $\rho_x$-invariant vectors is then $\C1$. As such by \eqref{eq:strongT}, for every $a\in \pi_x(A)''$ we have
    \begin{align}
        \label{eq:spgapdensesubalg1}
        \max_{g\in F}\|[\pi_x(u_g),a]\|_{2,\tau_x}& = \max_{g\in F}\|\rho_x(g)\hat{a}-\hat{a}\|_{L^2(A,x)}
        \\&\geq c\|\hat{a}- P_{\rho_x}\hat a\|_{L^2(A,x)}
        \\ &=c\|a-\tau_x(a)1\|_{2,\tau_x}.
    \end{align}
    Combining the identifications in \eqref{eq:identificationcompletion} with the above inequality shows that $\M$ definitely satisfies \cref{prop:equivchars} $(ii)$. Hence $\M = \overline{C^*(\Gamma)}^{T(C^*(\Gamma))}$ has uniform $w$-spectral gap.
\end{example}

In the remainder of this section we prove that having uniform $w$-spectral gap is really a local property of $\mathrm{W}^*$-bundles. This will now follow essentially from the fact that we only consider bundles over compact spaces.
\begin{proposition}
    \label{prop:localityofspgap}
    Let $\M$ be a factorial $\mathrm{W}^*$-bundle. Then $\M$ has uniform $w$-spectral gap if and only if for every $x\in X$ there is a closed neighbourhood $Y\subseteq X$ of $x$ such that $\M\lvert_Y$ has uniform $w$-spectral gap.
\end{proposition}
\begin{proof}
    The \say{only if part} follows immediately from surjectivity of the maps $\M\to \M\lvert_Y$, \Cref{prop:equivchars}$(ii)$ and the diagram \eqref{eq:tracediagram}. So it remains to proof the \say{if part}. By compactness we can cover $X$ by finitely many closed sets $Y_1,\dots, Y_k$ such that $\M\lvert_{Y_i}$ has uniform $w$-spectral gap for each $i=1,\dots,k$. Let $\varepsilon>0$. By \Cref{prop:equivchars} $(ii)$ we can find for each $i=1,\dots,k$ a finite subset $F_i = \left\{c_i^{(1)},\dots, c_i^{(n_i)}\right\}\subseteq \M\lvert_{Y_i}$ and $\delta_i>0$ such that if $a\in (\M\lvert_{Y_i})_1$ and $y\in Y_i$ satisfy $\max_{c\in F_i}\|[c,a]\|_{2,\tau_{y}}<\delta_i$ then $\|a-\tau_y(a)\|_{2,\tau_{y}}<\varepsilon$. Now let $b_i^{(j)}$ be a lift for $c_i^{(j)}$ for each $1\leq i\leq k, 1\leq j\leq n_i$. Then by the diagram \eqref{eq:tracediagram}, $F = \left\{b_{i}^{(j)}\mid 1\leq i\leq k, 1\leq j\leq n_i\right\}\subseteq \M$ and $\delta = \min\{\delta_i\mid 1\leq i\leq k\}$ do the trick for \Cref{prop:equivchars} $(ii)$.
\end{proof} 
\begin{corollary}
    \label{cor:locallytriv}
    Any locally trivial factorial $\mathrm{W}^*$-bundle $\M$ with full fibres has uniform $w$-spectral gap.
\end{corollary}
\begin{proof}
    This is immediate by combining \Cref{ex:trivialfull} with \Cref{prop:localityofspgap}.
\end{proof}

\section{Uniform $w$-spectral gap for subtrivial bundles}\label{sec:construction}
In this section we give a complete characterization of when a subtrivial bundle with fibres either a fixed full $\mathrm{II}_1$ factor $M$ with separable predual, or a fixed full subfactor $N\subseteq M$ has uniform $w$-spectral gap. As a consequence we can concretely construct factorial $\mathrm{W}^*$-bundles with all fibres isomorphic to some fixed full factor that are not locally trivial.

The following lemma will be the main technical tool for our characterization.
\begin{lemma}
    \label{lemma:techincallemma}
    Let $X$ be a compact Hausdorff space and $Y\subseteq X$ a closed subset. Suppose $(M,\tau)$ is a tracial von Neumann algebra and $N\subseteq M$ is a von Neumann subalgebra. Consider the subtrivial $\mathrm{W}^*$-bundle $\M = \{f\in C_\sigma(X,M)\mid f(x)\in N, \forall x\in Y\}.$ Let $(x_i)_{i\in I}$ be a net in $X$ and $\omega$ a cofinal ultrafilter on $I$. Let $x = \lim_{i\to \omega} x_i\in X$.
    \begin{enumerate}
        \item If $x\in Y$ then the following diagram commutes.
    \begin{equation}
    \begin{tikzcd}
	\M && {\prod_{i\in I}^\omega (\M, x_i)} \\
	N && {M^\omega}
	\arrow["{\iota^\omega}", from=1-1, to=1-3]
	\arrow["{\mathrm{eval}_x}"', two heads, from=1-1, to=2-1]
	\arrow["{(\mathrm{eval}_{x_i})^\omega}", hook', from=1-3, to=2-3]
	\arrow[hook, from=2-1, to=2-3]
\end{tikzcd}
    \end{equation}
    Furthermore $(\mathrm{eval}_{x_i})^\omega$ is a $*$-isomorphism if $x_i\in X\setminus Y$ for every $i\in I$.
    \item  If $x\notin Y$ then $(\mathrm{eval}_{x_i})^\omega$ is a $*$-isomorphism and the following diagram commutes.
    \begin{equation}
        \begin{tikzcd}
	\M && {\prod_{i\in I}^\omega (\M, x_i)} \\
	M && {M^\omega}
	\arrow["{\iota^\omega}", from=1-1, to=1-3]
	\arrow["{\mathrm{eval}_x}"', two heads, from=1-1, to=2-1]
	\arrow["(\mathrm{eval}_{x_i})^\omega", "\cong"', from=1-3, to=2-3]
	\arrow[hook, from=2-1, to=2-3]
\end{tikzcd}
    \end{equation}
    \end{enumerate}
\end{lemma}
\begin{proof}\leavevmode
    \begin{enumerate}[wide, labelwidth=!, labelindent=0pt]
        \item That $(\mathrm{eval}_{x_i})^\omega$ is a well-defined embedding follows from the fact that for each $i\in I$ and $a\in \M$ we have $\|a\|_{2,\tau_{x_i}} = \|\mathrm{eval}_{x_i}(a)\|_{2,\tau}.$ Surjectivity of $(\mathrm{eval}_{x_i})^\omega$ for the case where $x_i\in X\setminus Y$ for each $i\in I$ is implied by \Cref{prop:subtrivialbundles} and \Cref{ex:subtrivial}. 
        What remains to be proven is commutativity of the diagram. So let $a\in \M$, we have to show that $(a(x))_i = (a(x_i))_i$ in $M^\omega$. This is equivalent to 
        $$\lim_{i\to\omega} \|a(x)-a(x_i)\|_{2,\tau} =0,$$
        which follows from the fact that $a: X\to M$ is $\|\cdot\|_{2,\tau}$-continuous and $\lim_{i\to \omega}x_i = x$.
        \item Injectivity of $(\mathrm{eval}_{x_i})^\omega$ and commutativity of the diagram follows from exactly the same argument as in (1). What remains to be proven is surjectivity of $(\mathrm{eval}_{x_i})^\omega$. Let $(a_i)_i\in M^\omega$. Since $X\setminus Y$ is open and $x=\lim_{i\to \omega}x_i\in X\setminus Y$, the set
        \begin{equation}
            J = \{i\in I\mid x_i\in Y\setminus X\}
        \end{equation}
        is an element of $\omega$. By \Cref{prop:subtrivialbundles} we can find for each $i\in J$ a $b_i\in \M$ such that $\mathrm{eval}_{x_i}(b_i) = a_i$ and $\|b_i\|\leq \|a_i\|$. Now define $(c_i)_i\in \prod_i^\omega (\M,x_i)$ by
        \begin{equation}
            c_i =\begin{cases}
                b_i \quad \text{if} \quad i\in J\\
                1_\M  \text{ otherwise.}
            \end{cases}
        \end{equation}
        Then an easy computation shows $(\mathrm{eval}_{x_i})^\omega((c_i)_i)=(\mathrm{eval}_{x_i}(c_i))_i = (a_i)_i$ in $M^\omega$.
    \end{enumerate}
\end{proof}

\begin{theorem}
    \label{thm:theonlyone} 
    Let $X$ be a compact metrisable space and $Y\subseteq X$ a closed subset. Suppose $(M,\tau)$ is a full factor with separable predual and $N\subseteq M$ is a full subfactor. Consider the subtrivial $\mathrm{W}^*$-bundle $\N = \{f\in C_\sigma(X,M)\mid f(x)\in N, \forall x\in Y\}.$
    \begin{enumerate}
        \item If $Y$ is clopen, then $\N$ is locally trivial hence has uniform $w$-spectral gap.
        \item If $Y$ is not clopen, then $\N$ has uniform $w$-spectral gap if and only if $N$ is an irreducible subfactor of $M$ and $N$ has $w$-spectral gap in $M$.

    \end{enumerate}
\end{theorem}
\begin{proof}\leavevmode
    \begin{enumerate}[wide, labelwidth=!, labelindent=0pt]
        \item Suppose $Y$ is clopen. Then both $Y$ and $X\setminus Y$ are closed, \Cref{prop:subtrivialbundles} then implies $\N\lvert_Y \cong C_\sigma(Y,N)$ and $\N\lvert_{X\setminus Y}\cong C_\sigma(X\setminus Y, M)$ as $\mathrm{W}^*$-bundles. Hence $\N$ is locally trivial.
        \item Assume that $Y$ is not clopen. Firstly suppose that $\N$ has uniform $w$-spectral gap. Since $Y$ is not open, we can find a sequence $(x_n)_{n\geq 1} $ in $X\setminus Y$ that converges to some point $x\in Y$. Then for any cofinal ultrafilter $\omega$ on $\mathbb N$, we have $N'\cap M^\omega = \C1$ by combining \Cref{prop:equivchars} $(iv)$ with \Cref{lemma:techincallemma} (1). Thus we definitely have $N'\cap M = \C1_M$, that is $N\subseteq M$ is irreducible. Furthermore, by \Cref{prop:sepultrafilber} $N$ has $w$-spectral gap inside $M$.

        Conversely suppose $N\subseteq M$ is irreducible and that $N$ has $w$-spectral gap inside $M$. Let $(x_i)_{i\in I}$ be a net in $X$ and $\omega$ a cofinal ultrafilter on $I$. Since $M$ is assumed to be full, \Cref{lemma:techincallemma} combined with \Cref{prop:wspecgapincl} implies $\iota^\omega(\N)'\cap \prod_{i\in I}^\omega(\N,x_i) = \C1$. \Cref{prop:equivchars} $(iv)$ then implies $\N$ has uniform $w$-spectral gap.
    \end{enumerate}
\end{proof}
\begin{remark}
    Without the separability and metrisability assumptions \Cref{thm:theonlyone} is still partially true. Indeed, the case where $Y$ is clopen holds in general (that is for a general compact Hausdorff space $X$ and full factor $(M,\tau)$), by exactly the same proof as above.
    
    The case where $Y$ is not clopen is only partially true without separability and metrisability. Without these assumptions, the same argument as in the above proof shows that $N\subseteq M$ being irreducible and $N$ having $w$-spectral gap inside $M$ is sufficient for $\N$ having uniform $w$-spectral gap. Finally, a slight alteration of the argument (taking a net $(x_i)_i$ rather than a sequence), shows that $N\subseteq M$ being irreducible is also a necessary condition.
\end{remark}
\begin{remark}
    Notice that in \Cref{thm:theonlyone} we restrict our attention to closed subsets $Y\subseteq X$. A priori one could consider more general subtrivial bundles of the form $\N_0 = \{f\in C_\sigma(X,M)\mid f(x)\in N,\forall x \in Y_0\}$ where $Y_0$ is a general subset of $X$. However, since the elements of $C_\sigma(X,M)$ are $\|\cdot\|_{2,\tau}$-continuous and $N\subseteq M$ is $\|\cdot\|_{2,\tau}$-closed we have
    $\N_0 = \{f\in C_\sigma(X,M)\mid f(x)\in N, \forall x\in \overline{Y_0}\},$
    where $\overline{Y_0}$ is the closure of $Y_0$ in $X$.
\end{remark}
\begin{corollary}\label{cor:theexample}
    Let $(M,\tau)$ be a full factor that contains a non-irreducible subfactor $N\subseteq M$ that is isomorphic to $M$. Then the subtrivial bundle $$\N = \left\{f\in C_\sigma(\mathbb{N}\cup \{\infty\},M)\mid f(\infty)\in N\right\},$$ does not have uniform $w$-spectral gap. As a consequence $\N$ is a factorial $\mathrm{W}^*$-bundle with each fibre isomorphic to $M$ that is not locally trivial.
\end{corollary}
\begin{example}\label{ex:nonloctrivial}
    A very concrete example of a factorial $\mathrm{W}^*$-bundle that is not locally trivial but has all fibres isomorphic to some fixed factor can now be constructed as follows. Consider the free group factor $L\mathbb{F}_\infty$, it is well-known that $L\mathbb{F}_\infty \cong M_2(L\mathbb{F}_\infty)$, \cite{voiculescu}. The diagonal embedding $L\mathbb{F}_\infty\subseteq M_2(L\mathbb{F}_\infty)$  is of course not irreducible.
    As such the $\mathrm{W}^*$-bundle
    \begin{equation}
        \N = \{f\in C_\sigma(\mathbb{N}\cup\{\infty\}, M_2(L\mathbb{F}_\infty ))\mid f(\infty)\in L\mathbb{F}_\infty\subseteq M_2(L\mathbb{F}_\infty)\},
    \end{equation}
    is not locally trivial.
\end{example}
We end with some remarks.
\begin{remark}
    Since the hyperfinite $\mathrm{II}_1$ factor $R$ also satisfies $M_2(R)\cong R$, one could naively mimic the construction of \Cref{ex:nonloctrivial} with $R$ instead of $L\mathbb{F}_\infty$ and hope to obtain a $\mathrm{W}^*$-bundle that is not locally trivial. However, the bundle obtained in that way will(!) be trivial. Indeed by Ozawa's triviality result \cite{OzawaWbundle}, if there exists a non-trivial $\mathrm{W}^*$-bundle with all fibres isomorphic to $R$, then the base space must have infinite covering dimension. The (non)-existence of such a bundle is not known.
\end{remark}
\begin{remark}
    It would be interesting to see whether there are non-locally trivial $\mathrm{W}^*$-bundles with all fibres isomorphic to some fixed factor that do have uniform $w$-spectral gap. However, as of now, we know of no other obstruction for such bundles being locally trivial than the negation of the uniform $w$-spectral gap property.
\end{remark}
\begin{remark}
    Similarly as how Popa defines the two notions of having $w$-spectral gap and having spectral gap for inclusions of von Neumann algebras \cite{PopaSpgap}, one could also define a stronger version of having spectral gap for $\mathrm{W}^*$-bundles. That is replacing $(\M)_1$ with $\M$ in \Cref{def:uniwspectralgap}, we would call this property having \emph{uniform spectral gap} (again following the naming scheme set out in \cite{PopaSpgap}). Just as for inclusions of von Neumann algebras, one can use the example in the appendix of \cite{IoanaVaes} to show that also for $\rm{W}^*$-bundles the property of having uniform spectral gap is in general strictly stronger than having uniform $w$-spectral gap. However, many of the results in this paper also seem to have a uniform spectral gap version, this will be further investigated in subsequent work.
\end{remark}
\bibliographystyle{amsalpha}
\bibliography{bibfile}
\end{document}